\documentclass[12pt,reqno]{amsart}
\usepackage{graphicx}
\usepackage{amsmath}
\usepackage{amssymb}
\usepackage{amscd}
\usepackage{hhline}
\usepackage{mathrsfs}
 \usepackage{color}
 \usepackage{ifthen}
\usepackage{tikz}
\usetikzlibrary{matrix,arrows,decorations.pathmorphing}
\usepackage{framed}
\input{xstring}

\newtheorem{theoremb}{Theorem}
\newtheorem{theoremc}{Theorem\!\!}
\newtheorem{theoremd}{Theorem}

\newtheorem{theoremf}{Theorem}

\newtheorem{rk}[theoremf]{Remark}

\newtheorem{lem}[theoremd]{Lemma}

\newtheorem{prop}[theoremb]{Proposition}
\newcommand\bib[1]{\bibitem[#1]{#1}}

\newcommand\bond[1]{\draw (#1) -- +(1,0)}
\newcommand\tcirc[3]{
	\ifthenelse{\equal{#1}{w}}{\filldraw[fill=white,draw=black] (#2) circle (0.08);}{}%
	\ifthenelse{\equal{#1}{b}}{\filldraw[black] (#2) circle (0.08);}{}%
	\draw (#2) node[above=2pt] {#3};
	}
\newcommand\tcross[2]{
	\draw (#1) node[above=2pt] {#2};
	\draw (#1) ++(-0.12,-0.12)-- +(0.24, 0.24);
	\draw (#1) ++(-0.12, 0.12)-- +(0.24,-0.24);
	}
\newcommand\tstar[2]{
	\draw[color=red] (#1) node {\Large$*$};
	\draw (#1) node[above=2pt] {#2};
	}
\newcommand\tsquare[2]{
		\draw[semithick,color=blue] (#1) ++(-0.15,-0.15) rectangle ++(0.3,0.3);
		\tcross{#1}{#2};
		}
\newcommand\DDnode[3]{
\ifthenelse{\equal{#1}{w}}{\tcirc{w}{#2}{#3}}{}		% white - non-compact root (Satake diagram)
\ifthenelse{\equal{#1}{b}}{\tcirc{b}{#2}{#3}}{}		% black - compact root (Satake diagram)
\ifthenelse{\equal{#1}{x}}{\tcross{#2}{#3}}{}		% crossed root (corresponding to parabolic)
\ifthenelse{\equal{#1}{s}}{\tstar{#2}{#3}}{}		% starred root (my notation for sub-parabolic)
\ifthenelse{\equal{#1}{q}}{\tsquare{#2}{#3}}{}		% crossed square (Iw root)
}

\newcommand\com[1]{}
\newcommand\C{{\mathbb C}}
\newcommand\Cc{{\let\mathcal\mathscr\mathcal C}}

\newcommand\fa{\mathfrak{a}}

 \newcommand\fS{\mathfrak{S}}
 \newcommand\fU{\mathfrak{U}}
\newcommand\g{{\frak g}}

\newcommand\op[1]{\mathop{\rm #1}\nolimits}
\newcommand\ot{\otimes}
\newcommand\p{\partial}

\newcommand\R{{\mathbb R}}

\begin{document}

 \title{Submaximally symmetric CR-structures}
 \author{Boris Kruglikov}
 \date{}
 \address{Department of Mathematics and Statistics, University of Troms\o, Troms\o\ 90-37, Norway.
\quad E-mail: {\tt boris.kruglikov@uit.no}}

 \begin{abstract}
Hypersurface type CR-structures with non-degenerate Levi form on a manifold of dimension $(2n+1)$ 
have maximal symmetry dimension $n^2+4n+3$. We prove that the next (submaximal) possible dimension
for a (local) symmetry algebra is $n^2+4$ for Levi-indefinite structures and $n^2+3$ for Levi-definite structures 
when $n>1$. In the exceptional case of CR-dimension $n=1$, the submaximal symmetry dimension 3 was computed by E.\,Cartan.
 \end{abstract}

 \maketitle

%%%%%%%%%%%%%%%%%%%%%%%%%%%%%%%%%%%%%%%%%%%%%%%%%%%%%%%%%%%%%%%%%%%%%%%%%%%%
%0%
\section*{Introduction and Main Result}

Levi-nondegenerate Cauchy-Riemann (CR) structure of hypersurface type on a manifold $M$ of $\dim=2n+1$ consists of 
a contact distribution $C\subset TM$ and a complex structure $J$ on it that preserves the canonical conformal
symplectic structure $\omega$ on $C$ associated to the natural tensor $\bar\omega\in\Lambda^2C^*\otimes(TM/C)$,
$(X,Y)\mapsto[X,Y]\,\op{mod}C$: $\omega(JX,JY)=\omega(X,Y)$.

We will assume $J$ integrable, which is a condition analogous to vanishing 
of the Nijenhuis tensor: $[JX,JY]-[X,Y]=J([JX,Y]+[X,JY])$, $X,Y\in\mathcal{D}(M)$.
In this case under additional assumption of analyticity the CR-structure can be locally realized by
a real hypersurface in $\C^{n+1}$. 
Moreover the internal symmetries of CR-structures (vector fields
preserving both $C$ and $J$) are bijective with external symmetries (holomorphic vector fields
tangent to the hypersurface). We refer to these statements and other details to \cite{BER}.

Levi form of CR-structure $(C,J)$ is the conformal quadric $q(X,Y)=\omega(X,JY)$. It is non-degenerate 
(this is equivalent to nondegeneracy of $\bar\omega$, i.e.\ to $C$ being contact), and so has signature
$(2k,2n-2k)$ ($J$ is $q$-orthogonal; we fix $k\le\frac{n}2$ using a choice of the conformal factor). 
In the case $k=0$ we say $q$ is a Levi-definite form, elsewise $Q$ is Levi-indefinite.

E.\,Cartan \cite{Ca} for $n=1$ and N.\,Tanaka \cite{T} for $n>1$ showed (in the analytic case) that the symmetry (automorphism) 
group of a CR-structure of the considered type has maximal dimension $(n+2)^2-1$. The same estimate applies to
the Lie algebra of (local) symmetries. This maximal value is achieved only if the structure 
has CR-curvature $\kappa$ zero. In this and only in this case $(M,C,J)$ is locally equivalent to the round sphere 
$S^{2n+1}\subset\C^{n+1}$ with the induced CR-structure, or to another nondegenerate quadric (if the Levi form has indefinite signature). 
This seminal result is based on the construction of Cartan connection,
and it was re-proven by Chern-Moser (in the smooth case) in their investigation of CR-invariants \cite{CM}, and also in the 
general framework of parabolic geometry \cite{CS}.

What happens when $\kappa\neq0$ at some point of $M$? Here and below by
CR-curvature $\kappa$ one can understand either the Chern-Moser curvature or Tanaka's curvature or
the curvature of normal Cartan connection of the corresponding parabolic geometry, as well as the harmonic part of thereof.
The condition $\kappa\not\equiv0$ is independent of the choice. 

With this condition the maximal dimension of the symmetry algebra is called the sub-maximal symmetry dimension
(it was studied for many geometric structures, see \cite{Ko} and references therein).
Notice that dimension of the symmetry algebra is no less than that of the symmetry group\footnote{An example of difference: 
The automorphism group of flat $S^{2n+1}$ without a point is the parabolic subgroup 
$P_{1,n+1}\subset SU(1,n+1)$ of dimension $n^2+2n+2$, while the symmetry algebra is $\mathfrak{su}(1,n+1)$ 
of dimension $n^2+4n+3$.}, so the bound on the symmetry algebra is a stronger result.

For $n=1$ the sub-maximal symmetry dimension $\fS=3$ was established by E.\,Cartan (re-proved in \cite{KT} in the smooth case). 
For $n=2$ the answer comes from the results of A.\,Loboda \cite{L1,L2}: for Levi-definite case $\fS=7$ and
for Levi-indefinite case $\fS=8$ (actually he bounded the automorphism group, and assumed analyticity, but we 
will relax these assumptions without altering the values). For general $n$ (again under analyticity assumption)
the bound for the stability group $\op{Aut}_0(M,C,J)$ was achieved by V.\,Ezhov and A.\,Isaev \cite{EI}
(see also \cite{I}), but their theorem does not imply the submaximal value of the symmetry dimension.

 \begin{theoremc}
Consider a connected smooth Levi-nondegenerate hypersurface type CR-structure $(M,C,J)$ of CR-dimension $n>1$. 
If at one point the CR-curvature $\kappa\neq0$, then the symmetry algebra has dimension at most $\fS=n^2+3$ in the Levi-definite case and 
$\fS=n^2+4$ in the Levi-indefinite case. This estimate is sharp (= realizable).
 \end{theoremc}

% Notice that as a by-product of the computation we also obtain the bound $\dim\op{Aut}_0(M,C,J)\leq\fS-2n-1$,
% which rules out the first alternative in Theorems 1.1 and 1.2 of \cite{EI} at a regular point 
% (their result holds for all points, but if the above bound fails the point is umbilic).

Here by symmetries one can understand either 
(the germs of) local vector fields defined in a neighborhood of the point with $\kappa\neq0$ or global 
vector fields defined on the whole manifold be they incomplete or complete (in the last case this is equivalent to
consideration of the group of automorphisms). The statement holds true in all specifications.

\medskip
%\section*{Acknowledgements}

\textsc{Acknowledgement.}
I am grateful to Alexander Isaev for generous introduction to the literature 
and many inspiring discussions. I also thank Dennis The for useful comments.

%%%%%%%%%%%%%%%%%%%%%%%%%%%%%%%%%%%%%%%%%%%%%%%%%%%%%%%%%%%%%%%%%%%%%%%%%%%%
%1%
\section{Proof: Levi-indefinite case}\label{S1}

Let us begin by recalling that Levi-nondegenerate CR-structures of hypersurface type of CR-dimension $n$
(equivalently: $\dim M=2n+1$) are parabolic geometries of type $A_{n+1}/P_{1,n+1}$, see \cite{CS}.
Here we shall take instead of $A_{n+1}$ the real Lie group $SU(k+1,n-k+1)$ if the signature of the Levi form
is $(2k,2n-2k)$ (we assume $k\le n/2$). 

The Satake diagram of its Lie algebra $\mathfrak{su}(k+1,n-k+1)$ has $k+1$ arrows, $n-2k-1$ black nodes and
$2(k+1)$ white nodes\footnote{Both symmetric with respect to the Dynkin diagram automorphism (reflection), 
and with black nodes in the middle.} for $k<n/2$: 
 
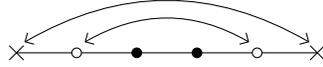
\[
{
 \begin{tiny}
 \begin{tikzpicture}[scale=0.8,baseline=-3pt]
\bond{0,0}; \bond{1,0}; \bond{2,0}; \bond{3,0}; \bond{4,0};
\DDnode{x}{0,0}{};  \DDnode{w}{1,0}{}; \DDnode{b}{2,0}{}; \DDnode{b}{3,0}{}; \DDnode{w}{4,0}{}; \DDnode{x}{5,0}{};
\node (B) at (1,0.1) {}; \node (C) at (4,0.1) {}; \path[<->,font=\scriptsize,>=angle 90] (B) edge [bend left] (C);
\node (A) at (0,0.1) {}; \node (D) at (5,0.1) {}; \path[<->,font=\scriptsize,>=angle 90] (A) edge [bend left] (D);
 \useasboundingbox (-.4,-.2) rectangle (5.2,0.55); % make bounding box bigger
 \end{tikzpicture}
 \end{tiny}
 }
 \]
and for $n=2k$ it has $k$ arrows and all $(2k+1)$ nodes white. We indicate by crosses the parabolic subalgebra.

For (regular normal) parabolic geometries a universal upper bound $\fU$ on dimension of the Lie algebra of symmetries 
of non-flat structures was derived in \cite{KT}: $\fS\leq\fU$. This works for all real smooth geometries, 
but sharpness is guaranteed (by abstract realization) only for complex analytic and real smooth structures of split type
(few exceptions to this were classified).
The number $\fU$ was computed for all non-rigid geometries (where non-flat structures exist) of type $G/P$,
where $G$ is simple and $P$ parabolic. 

The split-real case of CR-geometry (Cartan type $A_{n+1}/P_{1,n+1}$) is the geometry of Lagrangian 
contact structures \cite{CS}. The universal upper bound for the latter is $\fU=n^2+4$ ($n\ge2$), see \cite{KT} and the Appendix. 
Thus to prove our result for $\fS$ in the Levi-indefinite case we have to exhibit a model with that many symmetries.

 \begin{lem}
The CR-structure given as the real hypersurface in $\C^{n+1}$
 \begin{equation}\label{Lind}
\op{Im}(z_{n+1})=\op{Im}(z_1\bar{z}_2)+|z_1|^4+\sum_{k=3}^n\epsilon_k|z_k|^2
 \end{equation}
($\epsilon_k=\pm1$) has exactly $n^2+4$ independent symmetries ($n>1$).
 \end{lem}

 \begin{proof}
We express exterior symmetries as holomorphic vector fields:
 \begin{gather*}
H_1=z_1\,\p_{z_1}+3z_2\,\p_{z_2}+2\sum_{k=3}^nz_k\,\p_{z_k}+4z_{n+1}\,\p_{z_{n+1}},\\
H_2=i\,(z_1\,\p_{z_1}+z_2\,\p_{z_2}),\
H_k=i\,z_k\,\p_{z_k},\\
T_1=\p_{z_1}-z_2\,\p_{z_{n+1}}+4i\,z_1^2\,\p_{z_2},\
T_1'=i\,\p_{z_1}+i\,z_2\,\p_{z_{n+1}}+4\,z_1^2\,\p_{z_2},\\
T_2=\p_{z_2}+z_1\,\p_{z_{n+1}},\
T_2'=i\,\p_{z_2}-i\,z_1\,\p_{z_{n+1}},\\
T_k=\p_{z_k}+2i\epsilon_k\,z_k\,\p_{z_{n+1}},\
T_k'=i\,\p_{z_k}+2\epsilon_k\,z_k\,\p_{z_{n+1}},\\
T_{n+1}=\p_{z_{n+1}},\
S_1=z_1\,\p_{z_2},\\
S_k=2z_k\,\p_{z_2}+i\epsilon_k\,z_1\,\p_{z_k},\
S_k'=2i\,z_k\,\p_{z_2}+\epsilon_k\,z_1\,\p_{z_k},\\
R_{st}=\epsilon_sz_s\,\p_{z_t}-\epsilon_tz_t\,\p_{z_s},\
R_{st}'=i\,(\epsilon_sz_s\,\p_{z_t}+\epsilon_tz_t\,\p_{z_s}).
 \end{gather*}
The indices run as follows: $3\le k\le n$, $3\le s<t\le n$ (if the range is empty the corresponding terms
do not appear; notice also that $S_2$ does not appear). It is easy to check that these are 
symmetries\footnote{A holomorphic vector field $V$ is tangent to the real hypersurface $F=0$ if 
the following condition hold: $(V+\bar{V})\cdot F|_{F=0}=0$.},
are linearly independent and the totality of them is as required. 
Since the CR-curvature $\kappa$ is not zero, there can be no more symmetries.
 \end{proof}
 
 \begin{rk}
The model in the Lemma for $n=2$  is almost identical to that of \cite[p.146]{W}. It was shown
there that the symmetry is an 8-dimensional solvable Lie algebra. For $n>2$ the situation is different.
The Levi-decomposition of the algebra is
 $$
\op{sym}(\ref{Lind})=\mathfrak{su}(p,q)\ltimes\mathfrak{r},
 $$
where $p+q=n-2$, the signature is $(\epsilon_3,\dots,\epsilon_n)$, $\dim\mathfrak{r}=4n+1$.
The Levi factor is generated by $H_k-H_3$ ($3<k\le n$), $R_{st},R_{st}'$ ($3\le s<t\le n$),
and the radical $\mathfrak{r}$ is generated by $H_1,H_2,H_3+\dots+H_n$, $T_1,\dots,T_{n+1}$, $T_1',\dots,T_n'$, 
$S_1$, $S_3,\dots,S_n$, $S_3',\dots,S_n'$.
 \end{rk}

CR-structure (\ref{Lind}) is Levi-indefinite and can attain any non-positive (and non-negative) signature.
Thus we realized the upper bound on the symmetry and the theorem in this case is proven.

%%%%%%%%%%%%%%%%%%%%%%%%%%%%%%%%%%%%%%%%%%%%%%%%%%%%%%%%%%%%%%%%%%%%%%%%%%%%
%2%
\section{Proof: Levi-definite case}\label{S2}

Let us show that the upper bound $\fS\le\fU=n^2+4$ is not realizable for strictly pseudo-convex 
(or Levi-definite) CR-structures. Recall from \cite{Kr,KT} that the symmetry algebra acting on $M$
is filtered by a choice of a point on $M$, and the corresponding graded algebra $\fa\subset\g$. 

Recall \cite{CS} that in any parabolic geometry the curvatrue $\kappa$ of the normal Cartan connection has
a harmonic part $\kappa_H$ that uniquely determines flatness: $\kappa\equiv0$ $\Leftrightarrow$ $\kappa_H\equiv0$.
The harmonic curvature takes values in the cohomology space $H^2_+(\g_{-},\g)$, which it is a $\g_0$-module.

In our case, if $\kappa_H\neq0$ at the given point, then
$\fa=\fa_{-2}\oplus\fa_{-1}\oplus\fa_0\subset\mathfrak{su}(1,n+1)$,
where the latter algebra is $|2|$-graded (contact grading) in accordance with the choice of parabolic 
subalgebra $\mathfrak{p}_{1,n+1}\subset\g$. Absence of positive grading spaces in $\fa$ is due to
the prolongation rigidity phenomenon, see the Appendix.

In particular, the grading 0 space\footnote{The semi-simple part of $\g_0$ corresponds to the Dynkin diagram 
for $\g$ with crosses removed; the center has dimension of the number of crosses.} $\fa_0\subset\g_0=\mathfrak{u}(n)\oplus\R$ 
($\R$ is generated by the grading element), and as in \cite{KT} we observe 
that $\fa_0\subset\frak{ann}_{\g_0}(\kappa_H)\subset\g_0$. 

The largest dimension of the annihilator is obtained in the complex case, when the harmonic curvature $\kappa_H$ 
is given by  the lowest weight vector $\phi$ of the corresponding representation. In \cite{KT} we computed that 
for $A_{n+1}/P_{1,n+1}$ in the complex case $\dim\fa_0\le n^2-2n+3$.
 
 \begin{prop}
For Levi-definite CR-structures at the point with $\kappa_H\neq0$ the lowest weight vector $\phi$ is not real and 
$\dim\fa_0\le n^2-2n+2$.
 \end{prop}
 
 \begin{proof}
We give two independent proofs. At first let us notice that minus the lowest weight of the complexification
$H^2_+(\g_{-},\g)^\C= H^2_+(\g_{-}^\C,\g^\C)$ being computed via the Kostant algorithm 
(for this we refer to \cite{CS}) is as shown on the following Satake diagram 

 
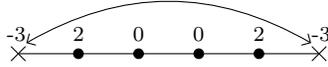
\[
{
 \begin{tiny}
 \begin{tikzpicture}[scale=0.8,baseline=-3pt]
\bond{0,0}; \bond{1,0}; \bond{2,0}; \bond{3,0}; \bond{4,0};
\DDnode{x}{0,0}{-3\,};  \DDnode{b}{1,0}{2}; \DDnode{b}{2,0}{0}; \DDnode{b}{3,0}{0}; \DDnode{b}{4,0}{2}; \DDnode{x}{5,0}{\,-3};
\node (B) at (0,0.1) {}; \node (E) at (5,0.1) {}; \path[<->,font=\scriptsize,>=angle 90] (B) edge [bend left] (E);
 \useasboundingbox (-.4,-.2) rectangle (5.2,0.55); % make bounding box bigger
 \end{tikzpicture}
 \end{tiny}
 }
 \]

We only need to notice that the lowest weight of the semi-simple part is non-trivial. The Satake diagram of the semisimple part of $\g_0$ 
is obtained by removing the crossed nodes: $(\g_0)_\text{ss}\simeq\mathfrak{su}(n)$. Since this is a compact Lie algebra, its real (nontrivial)
representation cannot have real lowest weight vector, and thus the orbit of minimal dimension (as in the complex case) is not realizable.
Equivalently, the annihilator of highest possible dimension is not possible for this real form, and so the dimension of $\fa_0$ drops at least by one.

The other proof is based on the fact, proven in \cite[Lemma 2.1]{IK} (we formulate a local version, which follows the same proof), 
that if a subalgebra $\tilde\fa_0\subset\mathfrak{u}(n)$ has dimension $\ge n^2-2n+3$, then it is either $\mathfrak{u}(n)$,
or $\mathfrak{su}(n)$, or $n=4$ and $\tilde\fa_0=\mathfrak{sp}(2)\oplus\R\subset\mathfrak{u}(4)$. 
Notice that the grading element $s$ cannot belong to $\fa_0$
(otherwise the symmetry algebra is graded and the curvature $\kappa_H$ vanishes at the given point \cite{KT}), so this algebra is isomorphic
to a subalgebra $\tilde\fa_0\subset\g_0/\R s=\mathfrak{u}(n)$. 

The first two of the above cases lead to $\kappa_H=0$ at the given point, as they have too big dimensions (by complex results of \cite{KT};
alternatively: the point has to be umbilic).
The last possibility for $n=4$ is not realized as $\mathfrak{sp}(2)$ does not preserve any nonzero vector in the curvature representation 
$\Lambda^2\g_{-1}^*\ot\g_0$ (this is directly checked in Maple\footnote{I am grateful to Henrik Winther for an assistance with this.}), 
which would again imply $\kappa_H=0$ contrary to the assumption.
 \end{proof}

Thus in Levi-definite case we improved the universal bound to $\fU\le n^2+3$, and to prove our result for $\fS$ 
we have to exhibit a model with that many symmetries.

 \begin{lem}
The CR-structure given as the real hypersurface in $\C^{n+1}$
 \begin{equation}\label{Ldef}
\op{Im}(z_{n+1})=\log(1+|z_1|^2)+\sum_{k=2}^n|z_k|^2
 \end{equation}
has exactly $n^2+3$ independent symmetries ($n>1$).
 \end{lem}

 \begin{proof}
We express exterior symmetries as holomorphic vector fields:
 \begin{gather*}
T_1=(1+z_1^2)\,\p_{z_1}+2i\,z_1\,\p_{z_{n+1}},\
T_1'=i\,(1-z_1^2)\,\p_{z_1}+2z_1\,\p_{z_{n+1}},\\
T_j=\p_{z_k}+2i\,z_k\,\p_{z_{n+1}},\
T_j'=i\,\p_{z_k}+2z_k\,\p_{z_{n+1}},\\
T_{n+1}=\p_{z_{n+1}},\
H_k=i\,z_k\,\p_{z_k}\\
R_{st}=z_s\,\p_{z_t}-z_t\,\p_{z_s},\
R_{st}'=i\,(z_s\,\p_{z_t}+z_t\,\p_{z_s}).
 \end{gather*}
The indices run as follows: $1\le k\le n$, $1<j\le n$, $2\le s<t\le n$ (if the range is empty 
the corresponding terms do not appear). It is easy to check that these are
linearly independent symmetries and the totality of them is as required.
Since the CR-curvature $\kappa$ is not zero, there can be no more symmetries.
 \end{proof}

 \begin{rk}
The model in the Lemma for $n=2$ is the last case of 4 submaximally symmetric CR-structures in 5D from 
\cite[Theorem~1]{L2}. In this case the Levi decomposition is the direct sum $\mathfrak{su}(2)\oplus\mathfrak{r}$,
where the radical is a 4D solvable Lie algebra with derived series of length 3. In the general case another simple factor
appears and the Levi-decomposition of the symmetry algebra is
 $$
\op{sym}(\ref{Ldef})=(\mathfrak{su}(2)\oplus\mathfrak{su}(n-1))\ltimes\mathfrak{r},
 $$
where $\dim\mathfrak{r}=2n$.
The Levi factor is generated by $T_1,T_1',T_{n+1}-H_1$ (the first simple factor) and 
$H_3-H_2,\dots,H_n-H_2$, $R_{st},R_{st}'$ (the second simple factor),
and the radical $\mathfrak{r}$ is generated by $H_2+\dots+H_n$, $T_2,T_2',\dots,T_n,T_n',T_{n+1}$.
 \end{rk}

CR-structure (\ref{Ldef}) is Levi-definite.
Thus we realized the upper bound on the symmetry and the theorem in this case is proven.

%%%%%%%%%%%%%%%%%%%%%%%%%%%%%%%%%%%%%%%%%%%%%%%%%%%%%%%%%%%%%%%%%%%%%%%%%%%%
%3%
\section{Concluding remarks}\label{S3}

In the above setup we considered integrable CR-structures (torsion-free parabolic geometries).
The statement of the theorem is also true for the general class of  partially integrable CR-structures (torsion allowed;
these are non-embeddable into $\C^{n+1}$ even in the analytic case) -- still the submaximal dimension $\fS$,
i.e.\ the maximal dimension of the symmetry algebra of a non-flat CR-structure, is given by the same value. 

Indeed, by \cite[Theorem 4.1.8]{KT} the value of $\fS$ does not exceed $\max_w\fU_w$, where $w$ belongs to
the Hasse diagram of weight two \cite{CS} enumerating the (lowest weights of)
irreducible cohomolgy components of $H^2_+(\g_{-},\g)$. 
For CR-geometry $A_\ell/P_{1,\ell}$, $\ell=n+1$, there are two irreducible components: 
$w_2=(1,\ell)$, corresponding to curvature that we considered, 
and $w_1=(1,2)\cup(\ell,\ell-1)$, corresponding to torsion. But the value $\fS_w=\fU_w$ is the same for $w_1$ and $w_2$, see  
\cite{KT} Table 11 (parabolic contact geometries). 

It is even true that the result (that the dimension of the symmetry of a non-flat structure does not exceed $\fS$) holds true for the general 
(not partially integrable) case, when one drops the condition $\bar\omega(JX,JY)=\bar\omega(X,Y)$, but we will not explore this here.

%%%%%%%%%%%%%%%%%%%%%%%%%%%%%%%%%%%%%%%%%%%%%%%%%%%%%%%%%%%%%%%%%%%%%%%%%%%%
%A%
\appendix
\section{Universal bound for the real structures}\label{A}

The symmetry algebra of any parabolic geometry of type $G/P$ is filtered (filtration depends on the choice of a point) 
and the corresponding graded algebra embeds $\fa\subset\g$. The universal upper bound of \cite{KT} for the symmetry algebra is as follows:
 $$
\fU\le\max\{\dim\fa^\phi\,:\,0\neq\phi\in H^2_+(\g_{-},\g)\},
 $$
where $\fa=\fa_{-}\oplus\fa_0^\phi\oplus\fa_1^\phi\oplus\dots$, $\fa_{-}=\g_{-}$, $\fa_0^\phi=\mathfrak{ann}(\phi)\subset\g_0$
and $\fa_i^\phi$ are $i$-th Tanaka prolongations of $(\g_{-},\fa_0^\phi)$. 

Moreover for complex or split-real cases we proved in \cite{KT} that the lowest dimension bound for $\fa^\phi$ 
is achieved if and only if $0\neq\phi\in H^2_+(\g_{-},\g)$ is the lowest\footnote{or highest, but we keep the lowest weight convention.} 
weight vector in a $\g_0$-submodule of $H^2_+$ (there were several exceptions for equality (sharpness), but the inequality from above holds always). 
This is true for other real cases.

 \begin{prop}
For any real parabolic geometry the submaximal symmetry dimension is bounded so:
 $$
\fU\le\max_{\mu}\dim_\C(\fa^{\phi_\mu})^{\C},
 $$
where $\phi_\mu$ is the lowest weight vector of the irreducible submodule $\mathbb{V}_\mu$ of minus lowest weight $\mu$ in the complexified 
$\g_0^\C$-module $H^2_+(\g_{-}^\C,\g^\C)$.
 \end{prop}

 \begin{proof}
Let $0\neq\kappa\in H^2_+(\g_{-},\g)$ be the curvature of the geometry. The algebra $\fa_0$ preserves it: $v\cdot\kappa=0$ $\forall v\in\fa_0$.
Let $\hat\kappa=(\kappa,0)\in H^2_-(\g_{-},\g)^\C$ be the complexification. Then the complexified algebra $(\fa_0)^\C$ preserves $\hat\kappa$.
In the complex case the inequality was proven in \cite{KT}, and the claim follows from $\dim_\R\fa_0=\dim_\C\fa_0^\C$ and commutation
of the Tanaka prolongation with complexification (because the prolongation is a linear algebra operation).
 \end{proof}

An important ingredient of the work \cite{KT} is prolongation-rigidity. A $G/P$ parabolic geometry is {\em prolongation-rigid\/} if for any
$0\neq\phi\in H^2_+(\g_{-},\g)$ the Tanaka prolongation of $\fa_0^\phi=\mathfrak{ann}(\phi)$ is $\op{pr}_+(\g_{-},\fa_0^\phi)=0$.

 \begin{prop}
If the complex geometry $G^\C/P^\C$ is prolongation-rigid, then so is the real geometry $G/P$.
 \end{prop}

 \begin{proof}
This again follows from commutation of the Tanaka prolongation with complexification and the fact that any $0\neq\phi\in H^2_+(\g_{-},\g)$
generates an element $0\neq\hat\phi\in H^2_+(\g_{-}^\C,\g^\C)$ in the complexification. In fact, the statement of our proposition 
follows at once from Corollary 2.4.8 and Proposition 3.1.1 of \cite{KT}.
 \end{proof}

Notice that if the geometry has real type, that is $G$ is a real Lie algebra, then its complexification has Dynkin diagram obtained by removing 
black colour and arrows from the Satake diagram, and computation of the cohomology by the Kostant algorithm is the same
(for complex type, when complexification is given by doubling of $\g$, the situation is more complicated -- see \cite{KMT}).
The CR-geometry has real type $SU(p,n+1-p)/P_{1,n}$, and so the universal bound is the same as in the complex (or split-real) case, and 
also the prolongation rigidity follows.

%%%%%%%%%%%%%%%%%%%%%%%%%%%%%%%%%%%%%%%%%%%%%%%%%%%%%%%%%%%%%%%%%%%%%%%%%%%%

\end{document}